\documentclass[12pt]{amsart}

\usepackage{amsmath} 
\usepackage{amsfonts}
\usepackage{amssymb}
\usepackage{amsthm}
\usepackage{latexsym}
\usepackage{color}
\usepackage{enumerate}
\usepackage{mathrsfs}
\usepackage{comment}
\usepackage{setspace}

\newtheorem{lemma}{Lemma}[section]
\newtheorem{thm}{Theorem}[section]
\newtheorem{prop}{Proposition}[section]
\newtheorem{coro}{Corollary}[section]

\newtheorem{remark}{Remark}[section]
\numberwithin{equation}{section}

\newcommand{\pr}{\partial}
\newcommand{\veps}{\varepsilon}

\def\tr{\textmd{tr}}

\def\Lap{\Delta}
\def\grad{\nabla}

\def\Vol{\textmd{vol}}
\def\L{\mathcal{L}}
\def\Div{\textmd{div}}
\def\dint{\displaystyle\int}
\def\R{\mathbb{R}}

\def\vol{\mathrm{vol}\,}
\def\R{\mathbb{R}}
\def\Scal{\mathrm{Scal}}

\newcommand{\be}{\begin{equation}}
\newcommand{\ee}{\end{equation}}
\newcommand{\bee}{\begin{equation*}}
\newcommand{\eee}{\end{equation*}}

\begin{document}

\title[]{Higher dimensional black hole initial data  with prescribed boundary metric}

\author{Armando J. {Cabrera Pacheco}$^1$}
\address[Armando Cabrera]{Department of Mathematics, University of Miami, Coral Gables, FL 33146, USA.}
\email{a.cabrera@math.miami.edu}
\thanks{$^1$Research partially supported by the National Council of Science and Technology of Mexico (CONACyT)}

\author{Pengzi Miao$^2$}
\address[Pengzi Miao]{Department of Mathematics, University of Miami, Coral Gables, FL 33146, USA.}
\email{pengzim@math.miami.edu}
\thanks{$^2$Research partially supported by Simons Foundation Collaboration Grant for Mathematicians \#281105.}

\begin{abstract} 

We obtain higher dimensional analogues of the results of Mantoulidis and Schoen in \cite{schoen-mant}.
More precisely, we show that (i)  any metric $g$ with positive 
scalar
curvature 
on the $3$-sphere $S^3$
can be realized 
as the induced  metric on the outermost apparent horizon 
of  a $4$-dimensional asymptotically flat manifold  with non-negative scalar curvature, 
whose ADM mass can be arranged  to be arbitrarily close to  the optimal value specified by the Riemannian Penrose inequality; 
(ii) any metric $g$ with positive scalar curvature on the $n$-sphere $S^n$, with $ n \ge 4 $,  
such that $(S^n, g)$ isometrically embeds into $ \R^{n+1}$ as a star-shaped hypersurface, 
can be realized as the induced  metric on the outermost apparent horizon 
of  an  $(n+1)$-dimensional asymptotically flat manifold  with non-negative scalar curvature, 
whose ADM mass can be made to be arbitrarily close to  the optimal value.
\end{abstract}
\maketitle

\markboth{Armando Cabrera and Pengzi Miao}{Black hole initial data  with prescribed boundary metric}

\section{Introduction and statement of results}

Recently, Mantoulidis and Schoen  \cite{schoen-mant} gave  an elegant construction of $3$-dimensional asymptotically flat manifolds with non-negative scalar curvature, 
whose ADM mass \cite{adm} is arbitrarily close to the optimal value determined by  the Riemannian Penrose inequality \cite{bray2001proof, huisken}, while  the intrinsic geometry on the  outermost apparent horizon is   ``far away" from being  rotationally symmetric. 
Their result   can be interpreted as a statement demonstrating the instability of the Riemannian Penrose inequality. The  construction in \cite{schoen-mant} is 
geometric and  can be outlined  as a two-step process: 
\begin{enumerate}
\item[1)]  Consider 
the set $\mathscr{M}^+$ consisting of metrics  on the $2$-sphere $S^2$ 
satisfying  $\lambda_1(-\Lap +K )>0$, where $K $ is  the Gaussian curvature. 
 Given any $ g \in \mathscr{M}^+$,
 construct a ``collar extension'' of $g$, which is a metric  of   positive scalar  curvature on the product 
 $[0,1 ] \times S^2$  such that the bottom boundary $ \{ 0 \} \times S^2$, having $g$ as the induced metric, 
is   outer-minimizing  while    the top boundary $\{ 1 \} \times S^2$ is metrically a round sphere 

\item[2)]  Pick any $m>0$ arbitrarily close  to $(A/16\pi)^{1/2}$, where $ A$ is the area of $(S^2, g)$.
Consider a $3$-dimensional spatial Schwarzschild manifold  of mass $m$ (which is scalar flat), 
deform 
it to have positive scalar curvature in a small region  near the horizon, and then glue it to the above collar extension 
by making use of the positivity of the scalar curvature. 
\end{enumerate}
In this way, Mantoulidis and Schoen 
in \cite{schoen-mant} 
constructed  asymptotically flat extensions of $(S^2,g)$ which have
non-negative scalar curvature and, outside a compact set,  
coincide with a spatial Schwarzschild manifold whose mass can be  arranged to be arbitrarily close to the optimal value 
  $(A/16\pi)^{1/2}$.
 
In recent years, there has been a growing interest in black hole geometry in higher dimensions.
Galloway and Schoen in \cite{galloway-schoen} obtained a generalization of Hawking's black hole theorem \cite{hawking} to higher dimensions. 
Their result shows that, in a  spacetime satisfying the dominant energy condition, cross sections  of the  event horizon  are of positive Yamabe type, i.e., they  admit metrics of positive scalar curvature. Bray and Lee in \cite{bray-lee} proved the Riemannian Penrose inequality for dimensions less than eight.  The  inequality asserts that the ADM mass  
$m_{_\textmd{ADM} }$ of   an  $(n+1)$-dimensional ($n<7$),  complete asymptotically flat manifold with non-negative scalar curvature, with boundary consisting  of  closed  outer-minimizing minimal hypersurfaces, satisfies 
\be\label{r-p-ineq}
m_{_\textmd{ADM} } \geq \frac{1}{2}\left(\dfrac{A}{\omega_{n}}\right)^{(n-1)/n},
\ee
where  $A$ is the volume of its boundary and  $\omega_{n}$ denotes  the volume of the standard unit $n$-sphere.

Motivated by the above results, in this work 
we give some higher dimensional analogues of the Mantoulidis-Schoen  theorem in \cite{schoen-mant}.
Given an integer $ n \ge 3$, denote  the $n$-dimensional sphere by $S^n$.
For simplicity, all metrics on $S^n$ below  are assumed to be smooth.
Our main results are

\begin{thm}\label{main}
Let $ g $ be a metric  with positive 
scalar
curvature on $ S^3$.
Denote  the volume of $(S^3, g)$ by $ \vol (g)$. 
Given any $m>0$ such that 
$ \omega_3  \left( 2 m  \right)^\frac32 >  \vol (g)$, 
there exists an asymptotically flat $4$-dimensional manifold $M^4$ with non-negative curvature such that
\begin{enumerate}[(i)]
\item $\pr M^4$ is isometric to $(S^3,g)$ and is minimal,
\item $M^4$, outside a compact set,   is isometric to a spatial Schwarzschild manifold of mass $m$, and 
\item $M^4$ is foliated by mean convex $3$-spheres which  eventually coincide with the rotationally symmetric $3$-spheres  
in the spatial  Schwarzschild manifold. 
\end{enumerate}
\end{thm}

\begin{thm} \label{main-2}
Given any $ n \ge 4 $,  let $ g$ be a metric  with positive scalar curvature on $ S^n$. 
Suppose  $(S^n, g)$  isometrically embeds 
into the Euclidean space $ \R^{n+1}$ as a star-shaped hypersurface. 
Denote  the volume of $(S^n, g)$ by $ \vol (g)$. 
Given  any $m>0$ such that $\omega_n (2m)^{n/(n-1)}>\vol(g)$,  there exists an asymptotically flat $(n+1)$-dimensional manifold  
$M^{n+1} $ with non-negative curvature such that
\begin{enumerate}[(i)]
\item $\pr M^{n+1} $ is isometric to $(S^n,g)$ and is minimal,
\item $M^{n+1}$, outside  a compact set,  is isometric to a spatial Schwarzschild manifold  of mass $m$, and
\item $M^{n+1}$ is foliated by mean convex $n$-spheres which  eventually coincide with the rotationally symmetric $n$-spheres  
in the spatial  Schwarzschild manifold.
\end{enumerate}
\end{thm}

We prove Theorems \ref{main} and \ref{main-2}
by following the two-step process in  Mantoulidis and Schoen's construction mentioned earlier. 
The key ingredient of our proof lies in the first step, in which we 
make use of
results of Marques \cite{Marques}, and of Gerhardt \cite{gerhardt} and Urbas \cite{urbas}, respectively, 
to construct the corresponding ``collar extensions".

This paper is organized as follows. 
In Section \ref{ricci-path},   
 we apply 
a fundamental result of Marques \cite{Marques} 
on deforming three-manifolds of positive scalar curvature
to join an initial metric $g$ of positive scalar curvature on $S^3$  to a  round metric
via a smooth path of metrics of positive scalar  curvature. 
In Section \ref{star-shaped},  we apply one type  of inverse curvature flow in $ \R^{n+1}$, 
studied by Gerhardt \cite{gerhardt} and also by Urbas \cite{urbas},
to connect the metric $g$  on $S^n$ satisfying the condition in Theorem \ref{main-2}  
to a  round metric, via a smooth path of metrics of  positive scalar curvature. 
In Section \ref{S-M},  we carry out Mantoulidis and Schoen's construction in higher dimensions $n \ge 3$.
In Section \ref{sec-proof},  we prove Theorems \ref{main} and \ref{main-2}. 

\vspace{.2cm}

{\bf Acknowledgments}. We want to give our sincere thanks to F. C. Marques for suggesting the proof of Proposition \ref{prop-smoothing}.

\section{Smooth paths in $\Scal^+ (S^3) $} \label{ricci-path}

Let  $\Scal^+ (S^3) $
denote  the set of smooth  metrics  with positive scalar  curvature on $S^3$.
Given  $g\in \Scal^+(S^3)$,
the first step to perform a collar extension of $g$, needed in the proof of Theorem \ref{main},
 is to connect $g$ to a round metric on $S^3$ via a smooth path 
 in $ \Scal^+(S^3)$. 
We will achieve this by first applying the  result of Marques \cite{Marques} to obtain 
  a continuous path,  and then  by mollifying this continuous path to obtain a smooth path. 

We begin with a general path-smoothing procedure, 
 suggested to us by Marques \cite{Marques-p}.
Let $ M$ be an $n$-dimensional, $ n \ge 2$, smooth closed manifold. 
Let $ \mathcal{S}^k (M)$ denote the  space of $C^k$ symmetric $(0,2)$ tensors
on $M$ endowed with the $C^k$ topology. Here $ k \ge 0 $ is either  an integer  or $ k = \infty$. 
Let $ \mathcal{M}^k (M)$ be the open set in $\mathcal{S}^k (M)$ consisting of Riemannian  
metrics.  Given any $g \in \mathcal{M}^k (M)$ with $ k \ge 2$,  let $ R(g)$ denote the scalar 
curvature of $g$.

\begin{lemma} \label{lma-R-1}
Let $ \{ g (t) \}_{ t \in [0,1] }$ be a continuous path in $  \mathcal{M}^k (M)$, $k \ge 2$.
Suppose $ R (g(t)) > 0  $ for each $t$. 
Then there exists a constant $ \epsilon > 0 $ such that, for any  $g \in \mathcal{M}^k (M)$, 
if $ || g - g (t) ||_{C^2} < \epsilon  $ for some $ t \in [0,1]$, then $ R(g)> 0$. 
\end{lemma}

\begin{proof}
Suppose the claim is not true. Then for any integer $j > 0 $, there exists a metric  $g_j \in \mathcal{M}^k (M)$  and some $t_j \in [0,1]$
such that 
$  || g_j - g ({t_j}) ||_{C^2} < \frac{1}{j}$
while $ R(g_j) \le 0 $ somewhere on $M$.
Passing to a subsequence, we may assume $\lim_{j \rightarrow \infty} t_j = t_* $ for some point $t_* \in [0,1]$.
Since $ R(g ({t_*}) ) > 0 $, there exists  $ \epsilon_0 > 0 $ such that
 if  $ g   \in \mathcal{M}^k (M)$ and 
$ || g - g ({t_*})  ||_{C^2} < \epsilon_0 $, then $ R(g) > 0 $.
For  large $j$, by the continuity of $\{ g(t) \}$ in $\mathcal{M}^k(M)$,
we now have $ || g_j - g ({t_*}) ||_{C^2} < \epsilon_0 $, hence $ R (g_j) > 0 $ 
which is a contradiction. 
\end{proof}

\vspace{.1cm}

\begin{prop} \label{prop-smoothing} 
Let $ \{ g (t) \}_{ t \in [0,1] }$ be a continuous path in $  \mathcal{M}^k (M)$, $ k \ge 2$.
Suppose    $ R (g (t) ) > 0  $ for each $t$. 
Then there exists a smooth path $\{ h (t)  \}_{ t \in [0,1]} $ in  $  \mathcal{M}^k (M)$ satisfying 
$h (0) = g (0)$, $ h (1) = g (1)$ and $ R (h(t) ) > 0$ for all $t$.
\end{prop}

\begin{proof}
Let $ \epsilon > 0 $ be the constant given by Lemma \ref{lma-R-1}. Since the map
$ t \mapsto g_t \in \mathcal{M}^k (M) $
is continuous on $[0,1]$, there exists $ \delta > 0 $ such that, 
if  $ t$, $t' \in [0,1]$ with $ | t - t' | < \delta$, then $|| g(t) - g (t') ||_{C^2} < \epsilon $. 

Let $ t_0 = 0 < t_1 < \ldots < t_{m-1} < t_m = 1 $  be a sequence of points such that
$  | t_{i-1} - t_{i} | < \delta $, $\forall $ $i = 1, \ldots, m$. On each $ [t_{i-1}, t_i]$, define 
\be \label{eq-linear-path}
 h^{(i)} ({t}) = \frac{ t - t_{i-1}}{ t_i - t_{i-1}} g ({t_i})  +  \frac{ t_i - t}{ t_i - t_{i-1}}  g ({t_{i-1}) } . 
 \ee
Clearly, $ h^{(i)} (t) \in \mathcal{M}^k (M)$ and  $\{ h^{(i)} (t) \}_{t \in [t_{i-1}, t_i]}$ is a smooth path 
in $  \mathcal{M}^k (M) $.
Moreover, 
\be \label{eq-h-and-g-ti}
\begin{split}
 || h^{(i)} (t)  - g (t_{i-1}) ||_{C^2}  = & \ \frac{  t - t_{i-1} }{t_i - t_{i-1}} || g (t_{i} ) - g (t_{i-1} ) ||_{C^2} < \epsilon, \\
 || h^{(i)} (t)  - g (t_i) ||_{C^2}  = & \ \frac{ t_i - t}{t_i - t_{i-1}} || g (t_{i} ) - g ({t_{i-1}}  )||_{C^2} < \epsilon.
\end{split}
\ee
In particular,  $ R (h^{(i)} (t) ) > 0 $ by Lemma \ref{lma-R-1}.
Let $ \{ \hat h (t) \}_{t \in [0,1]}$ be the path of metrics  obtained by 
replacing $ \{ g (t)  \}$ by $ \{ h^{(i)} (t) \} $ on each $[t_{i-1}, t_i]$.  
Then $ \{ \hat h (t) \}_{[0,1]}$ satisfy all the properties desired for $\{h(t)\}_{[0,1]}$ except that 
it is not smooth at the points  $t_1, \ldots, t_{m-1}$. 

To complete the proof, we will mollify $ \{ \hat h (t) \}_{[0,1]}$ near each ``corner"  $t_i$, $ 1\le i \le m-1$. 
 We demonstrate the construction on  $\left( \frac{t_0+ t_1}{2},  \frac{t_1 + t_2}{2} \right) $ as follows. 
Let $ \phi = \phi (s) $ be  a smooth function with compact support in $(-1, 1)$ such that 
$ 0 \le \phi \le 1$, $ \int_{-\infty}^\infty \phi (s) d s =  1 $ and 
\be \label{eq-symmetry}
\phi (s) = \phi (-s). 
\ee
Let $ \sigma > 0 $ be a fixed constant such that 
$ \sigma < \min \left\{ \frac{ t_i - t_{i-1}}{4} \ | \ i =1, \ldots, m \right\} $.  
Let  $ \phi_\sigma (s) = \sigma^{-1} \phi ( \frac{s}{\sigma} )$.
For each   $ t \in \left( \frac{t_0+ t_1}{2},  \frac{t_1 + t_2}{2} \right) $,  define
\be
\begin{split}
h^{(1)}_\sigma  (t) = & \  \int_{-\sigma}^\sigma  \hat h ({ t-s}) \phi_\sigma (s) d s  \\
= & \ \int_{0}^{1}  \hat h  (u)  \phi_\sigma ( t - u ) d u .
\end{split}
\ee
Evidently, $h^{(1)}_\sigma (t)$  lies in $\mathcal{S}^k(M)$ and is smooth in $t$.
 By the convexity of  $ \mathcal{M}^k(M)$ in $\mathcal{S}^k(M)$,  $ h^{(1)}_\sigma (t) $ indeed  lies  in $ \mathcal{M}^k(M)$. 
Moreover, 
\be
h^{(1)}_{ \sigma} (t)  - g(t_1) =  \int_{-\sigma}^\sigma  \left[ \hat h  ({ t-s}) - g(t_1) \right]  \phi_\sigma (s) d s  ,
\ee
which combined  with \eqref{eq-h-and-g-ti} implies 
\be
|| h^{(1)}_{ \sigma} (t)  - g(t_1) ||_{C^2} < \epsilon . 
\ee
Hence, $ R (h^{(i)}_{\sigma}(t) ) > 0 $ by Lemma \ref{lma-R-1}.
Now suppose $ t \in \left( \frac{t_0+ t_1}{2}, \frac{t_0 + 3 t_1}{4} \right) $. 
Then $(t - \sigma, t + \sigma) \subset (t_0, t_1)$. 
Therefore, by \eqref{eq-linear-path} and \eqref{eq-symmetry}, 
 \be
\begin{split}
h^{(1)}_{ \sigma} (t) = & \  \int_{\R^1}  \left[  \frac{ (t-s) - t_{0}}{ t_1 - t_{0}} g_{t_1}  +  \frac{ t_1 - (t-s) }{ t_1 - t_{0}}  g_{t_{0}}  \right]  \phi_\sigma (s) d s  \\
= & \ \int_{\R^1}  \left[  h^{(1)} (t)   +  \frac{ s }{ t_1 - t_{0}}  ( g_{t_{0}} - g_{t_1} )   \right]  \phi_\sigma (s) d s  \\
= & \ h^{(1)}(t) .  
\end{split}
\ee
Similarly, for $ t \in \left( \frac{ t_1 + 3 t_2}{4} , \frac{ t_1 + t_2}{2} \right) $, we have
$
h^{(1)}_{ \sigma} (t)  = h^{(2)}(t). 
$
In other words, the path $ \{  h^{(1)}_{\sigma}(t) \}$ coincides with $ \{ \hat h (t) \}$ 
near $\frac{t_0 +  t_1}{2} $ and $  \frac{t_1 +  t_2}{2}  $.

Applying the above  construction on each  
$ I_i =  \left( \frac{t_{i-1}+ t_i}{2},  \frac{t_{i} + t_{i+1}}{2} \right) $ 
to obtain $ h^{(i)}_\sigma (t)$ and  then   replacing $\hat h (t) $ by $ h^{(i)}_\sigma (t) $ 
on $I_i$,   $ i = 1, \ldots, m-1$, 
we  obtain a smooth path $\{h (t) \}_{t\in[0,1]}$ meeting  all conditions required.
 This completes the proof. 
\end{proof}

Now we state the result of Marques \cite[Corollary 1.1]{Marques}, asserting the path connectedness of the space $\Scal^+(S^3) \subset \mathcal{M}^\infty (S^3)$.

\begin{thm}[\cite{Marques}] \label{thm-Marques}
Given any  $ g \in \Scal^+ (S^3)$, there exists a continuous path $ \{ g(t) \}_{t\in[0,1]}$ in 
$\Scal^+ (S^3)$  connecting $g$ to a round metric on $S^3$.
\end{thm}

The following corollary follows directly from Marques' theorem, Theorem \ref{thm-Marques}, and Proposition \ref{prop-smoothing}. 

\begin{coro}\label{coro-smooth}
Given any  $ g \in \Scal^+ (S^3)$,   
there exists a smooth  path $ \{ h(t) \}_{t\in[0,1]}$ in 
$\Scal^+ (S^3)$  connecting $g$ to a round metric on $S^3$.
\end{coro}

\section{Smooth paths in $\Scal^+_* (S^n) $} \label{star-shaped} 

In this section, we make preparations for the proof of Theorem \ref{main-2}.
For  $ n \ge 2$, 
let  $ \Scal^+_* (S^n)$ denote   the set of smooth 
metrics $g$  with positive scalar curvature 
on $ S^n$
such that $(S^n , g)$  isometrically embeds 
in $ \R^{n+1}$ as a star-shaped hypersurface.
When $ n =2$, 
by the results in \cite{nirenberg, pogorelov}, 
$\Scal^+_* (S^2) $ agrees with the set of metrics on $S^2$ with  positive Gaussian curvature.

Below,  we focus on $ n \ge 3$. 
Given $ g \in \Scal^+_* (S^n) $, by applying the work  of Gerhardt \cite{gerhardt} and Urbas  \cite{urbas}, 
we 
verify  
that $ g$ can be connected to a round metric on $S^n$  via a smooth path in $ \Scal^+_*(S^n)$. 
We begin  with a  lemma  that ensures  the positivity of the mean curvature of  
the  embedding. 

\begin{lemma} \label{lma-pH}
Let  $ \Sigma $ be a closed hypersurface in $ \R^{n+1}$.  Suppose the induced metric on $ \Sigma$
 has positive scalar curvature.  Then,  the mean curvature  of $ \Sigma$ with respect to the outward normal 
is everywhere positive.
\end{lemma}

\begin{proof}
Let $ R$ and $H$ be the scalar curvature and the mean curvature of $ \Sigma$, respectively. 
By the Gauss equation, we have
$ R = H^2 - | \Pi |^2,  $
where $ \Pi $ is the 2nd fundamental form of $\Sigma$ in $ \R^{n+1}$. Hence, $ R > 0 $ implies $ H^2 > 0$. 
Since $ \Sigma$ is closed, there exists a point on $ \Sigma$ at which $ H \ge 0 $, and hence, we conclude that $ H > 0$ everywhere on $ \Sigma$.
\end{proof}

Given a closed  hypersurface $ \Sigma \subset  \R^{n+1}$, 
let $ \kappa_1, \ldots, \kappa_n $ denote the principal curvatures of $ \Sigma$ with  respect to the outward normal
at each point. 
For $ 1 \le k \le n $,  define the normalized $k$-th mean curvature $\sigma_k$ of $\Sigma$ by
\be
\sigma_k= \frac{1}{{n\choose k}} \sum_{ 1 \le i_1 < i_2 \ldots < i_k \le n }   \kappa_{i_1} \cdots \kappa_{i_k}   .
\ee
Clearly, $ \sigma_1 $ and $ \sigma_2 $ are related to the 
 usual mean curvature $H$ and the scalar curvature $ R$  of $ \Sigma$, respectively, by
\be
H = n \sigma_1 , \ \ 
R = n (n-1) \sigma_2  .  
\ee
We say that a smooth map
$ X: \Sigma \times [0, \infty) \longrightarrow \R^{n+1} $
is  a solution to the $ \sigma_1 / \sigma_2 $ flow if $X$ satisfies 
\be \label{eq-flow-hr}
\frac{\partial X}{\partial  t} 
= \frac{ \sigma_1 } { \sigma_2 }  \nu 
=  \frac{(n-1)H}{R} \nu , 
\ee
where  $\nu$ is the outward unit normal to $ \Sigma_t = X (\Sigma, t)$ and $ H$ and $ R$ are the mean curvature and the 
scalar curvature of $ \Sigma_t $, respectively. 
By definition,  if the $ \{ \Sigma_t \}$ arise from a smooth solution to \eqref{eq-flow-hr},  
$ R $ does not vanish along $ \Sigma_t$, hence must be positive. Consequently,  
by Lemma \ref{lma-pH}, $ H$ must be  positive along $ \Sigma_t$. Hence, the surfaces $ \Sigma_t $ are moving outward. 
The $\sigma_1/\sigma_2 $ flow  is one type of the inverse curvature flows in $ \R^{n+1}$ studied by Gerhardt in \cite{gerhardt} and independently
 by Urbas in \cite{urbas}.
In particular,  the following theorem
is a special case of the general result   proved in \cite{gerhardt} and \cite{urbas}.

\begin{thm} [\cite{gerhardt, urbas}] \label{Gerhardt-Urbas}
Let $ X_0 : S^n \rightarrow \R^{n+1}$ be a smooth embedding such that  
$ \Sigma_0  = X_0 (S^n) $ is star-shaped with respect to a point $P_0  \in \R^{n+1}$.
If $ \Sigma_0$ has positive scalar curvature and positive mean curvature, then 
the $\sigma_1 / \sigma_2 $  flow 
\be \label{eq-HR-g}
 \frac{\partial X}{\partial t}  =  \frac{(n-1) H}{R} \nu,
 \ee
with the initial condition 
$  X (\cdot , 0) =  X_0 (\cdot)    $
has a unique smooth solution $ X: S^n \times [0, \infty) \rightarrow \R^{n+1}$.
In particular, each $ \Sigma_t = X(S^n, t)$ has positive scalar curvature $ R$ and positive mean curvature $H$. 
Moreover,  the rescaled surface $  e^{-t} X(S^n, t) $ converges  to a round sphere centered at $P_0$
  in the $ C^\infty$ topology as $ t \rightarrow \infty$.
\end{thm}

In what follows, we  
argue  that  
the proof of Theorem \ref{Gerhardt-Urbas} 
in   \cite{gerhardt, urbas}
indeed provides  a smooth path of  metrics, with positive scalar  curvature,    connecting  any  $ g \in \Scal^+_* (S^n) $ to a round metric on $S^n$. 
We follow the notations used in \cite{urbas}. 
Identifying  $ S^n $ with   the unit sphere $ \{ x \in \R^{n+1} \ | \ | x | = 1 \}  $ and 
assuming $ X_0 : (S^n , g) \rightarrow \R^{n+1}$ is an isometric embedding such that 
 $X_0 (S^n)$ is star-shaped with respect to the origin (modulo  a diffeomorphism on $S^n$), we can  write $X_0$ as 
\be 
X_0 (x) = \rho_0 (x) x , \ x \in S^n . 
\ee
Here $\rho_0 : S^n \rightarrow \R^+ $ is a smooth positive function on $S^n$, referred to as the {\em radial function} 
representing $ \Sigma_0 $ in \cite{urbas}.
For each $ t > 0$, 
the surface $ \Sigma_t  $ in Theorem \ref{Gerhardt-Urbas}  is  then  given by the graph of a  function $ \rho (\cdot, t)$ over $ S^n$, where
\be
 \rho (\cdot, \cdot) : S^n \times [0, \infty) \longrightarrow \R^+ 
 \ee
is a smooth function solving  the  parabolic equation (2.8) in \cite{urbas},  i.e.,
\be
\frac{\partial \rho }{\partial t }  = \frac{ \left( \rho^2 + | \nabla \rho |^2 \right)^\frac12}{ \rho F( a_{ij} ) }, 
\ee
with the initial condition $ \rho (\cdot, 0) = \rho_0 $.
Here, $ \nabla $  denotes the gradient on $S^n$ with respect to the standard metric  and 
$F (a_{ij})$ is given by equation (2.9) in \cite{urbas}, which is simply the expression of 
 $ {\sigma_2} / {\sigma_1} $  in terms of $ \rho(\cdot, t)$.
For each $t$, one can rescale    $\rho $ to define  $ \tilde \rho (\cdot, t)  = e^{-t} \rho (\cdot, t)$.  
Then  $\tilde \rho (\cdot, t) $   satisfies 
\be \label{eq-pde-trho}
\frac{\partial \tilde \rho }{\partial t }  = \frac{ \left( {\tilde \rho}^2 + | \nabla \tilde  \rho |^2 \right)^\frac12}{ \tilde \rho F( \tilde{a}_{ij} ) }
- \tilde \rho ,
\ee
where $ F( \tilde a_{ij} )$ is the expression of $\sigma_2 / \sigma_1$ associated to  the graph of $ \tilde \rho (\cdot, t)$ over $S^n$ 
(see equation (3.28) in \cite{urbas}). 
The following 
estimates on $ \tilde \rho (\cdot, t)  $ and the   convergence of $\tilde \rho (\cdot, t) $ as $t \rightarrow \infty$  are given by  (3.38), (3.39) and (3.40) in \cite{urbas}: 
\begin{enumerate}
\item[a)] There exist positive constants $C$ and $\gamma$ such that 
\be
\max_{ S^n} | \tilde \rho ( \cdot, t) - \rho^* | \le C e^{ - \gamma t} .
\ee
Here, $ \rho^* > 0 $ is some constant.

\item[b)] For any positive integer $ k$ and any constant $ \tilde{\gamma} \in (0, \gamma)$, 
there exists a positive constant $ C_k = C_k (\gamma, \tilde \gamma ) $ such that
\be
\int_{S^n} | \nabla^k \tilde \rho (\cdot, t) |^2 \le C_{k} e^{- \tilde \gamma t} . 
\ee

\item[c)] Given any two  integers $l \ge 0$ and $ k > l + \frac{n}{2}$, 
there exists a positive  $ C = C(k, l)$ such that 
\be
|| \tilde \rho (\cdot, t) - \rho^* ||_{C^l (S^n) } \le C 
\left[   \int_{S^n} | \nabla^k \tilde \rho (\cdot, t) |^2
+ \int_{S^n}| \tilde \rho ( \cdot, t) - \rho^* |^2 \right]^\frac12 . 
\ee

\end{enumerate}
It follows directly from a), b) and c) that   there exists a constant $ \delta > 0 $ (say $ \delta = \frac12 \gamma$)
such that, for any integer $ l \ge 0 $,
\be \label{eq-est-spatial}
|| \tilde \rho (\cdot, t) - \rho^* ||_{C^l (S^n) } \le C  e^{-\delta t}. 
\ee
This, combined with the PDE  \eqref{eq-pde-trho},  in turn   implies, for any integer  $ k \ge 1$, 
\be \label{test-starshape}
\left\| \frac{\partial^k \tilde \rho }{\partial t^k}  \right\|_{C^l (S^n)}  \le   C   e^{-\delta t},
\ee
for some constants $C$.

Now, we can define  a  path of metrics in $ \Scal^+_* (S^n)$
 connecting $ g$ to a round metric $g^*$ that corresponds to  a round sphere in $ \R^{n+1}$ 
of radius $ \rho^*$ as follows. 
Define $ \Phi_t : S^n \rightarrow \R^{n+1} $  by
$ \Phi_t (x) = \tilde \rho(x, t) x $
and let 
\be
 g(t)  = \Phi_t^* ( g_{_E}), 
 \ee
where $ g_{_E}$ is the Euclidean metric on $ \R^{n+1}$. 
Theorem \ref{Gerhardt-Urbas} guarantees that  $ g(t)$ has positive scalar curvature.
Moreover,  given any integers $l$ and $k$,  
it follows  from \eqref{eq-est-spatial} and \eqref{test-starshape} that 
\be \label{metric-estimate}
\| g(t) - g^* \|_{C^l (S^n) } \le   C   e^{-\delta t} 
\ee
and
\be \label{derivative-estimate}
\left\| \frac{\partial^k}{\partial t^k} g(t) \right\|_{C^l (S^n)}  \le   C   e^{-\delta t} .
\ee

We make a change of variable $ t = t(s)$  to view  the metrics $ \{ g(t) \}$ 
as a new family of metrics $\{ h(s ) \}$ defined on the  finite interval $[0,1]$.
Specifically, let 
\be \label{path-finite-int}
t(s)=\dfrac{1}{(s-1)^2}-1,\mbox{ (then $s=1-\dfrac{1}{\sqrt{1+t}}$)}
\ee
and define
\be \label{eq-hs-s}
h(s)=\left\{ \begin{array}{cl} g(t(s)) & \textmd{when $s\in[0,1)$}   \\ g^* & \textmd{when $s=1$}\end{array} \right.,
\ee
which is continuous by \eqref{metric-estimate}. Using the exponential decay estimate of the derivatives in \eqref{derivative-estimate}, one concludes  that the metric defined by
\be  
H = d s^2 + h(s),
\ee
is smooth on $ I \times S^n$ and it satisfies that $h(0) = g$,  $h(s)$ is a metric of positive scalar curvature on $S^n$  for all
$s \in [0, 1]$,  and $ h(1)$ is a round metric.

\section{Mantoulidis-Schoen  construction in higher dimensions}\label{S-M}

In this section, we recall the construction of Mantoulidis and Schoen from Section 1 and 2 in \cite{schoen-mant}.
Though stated for dimension $n=2$, many of their arguments apply in a straightforward manner to  higher dimensions
 $ n \ge 3$.  
For readers' convenience, we include the proof of all lemmas stated  below.

\subsection{Deformations on $I \times S^n$}

\begin{lemma}\label{nice-path}
Suppose $ \{ h (t) \}_{0 \le t \le 1} $ is a family of metrics of  positive scalar curvature  on $S^n$   such that 
\begin{itemize}
\item $ h(1) $ is a round metric.
\item $H=dt^2+h(t)$ is a smooth metric on $I\times S^{n}$, where $I = [0, 1]$.
\end{itemize}
Also, suppose that $ \vol (h(t))$, 
 the volume of $(S^n, h(t) )$, is a constant independent of $t$.  
Then, there exists a smooth metric $G = d t^2 + g(t)  $ on $I\times S^{n}$ satisfying
\begin{enumerate}[(i)]
\item $g(0) $ is isometric to $h(0)$ on $S^n$,
\item $g(t)$ has positive scalar curvature  $\forall \ t \in I$, 
\item $g(1)$ is round, $ g(t) = g(1)$ $\forall $  $t\in [1/2,1]$,  and 
\item $\frac{d}{dt}dV_{g(t)}=0$ for all $t\in [0,1]$. Here, $ d V_{g(t)}$ is the volume form of $g(t)$ on $S^n$. 
\end{enumerate} 
\end{lemma}

\begin{proof}
Choose  a smooth monotone function $\zeta$ on $[0,1]$ such that $\zeta(0)=0$ and $\zeta(t)=1$, $t \in [1/2,1]$.
Consider  $ h (\zeta (t) )$, $ t \in [0,1]$.
This new path $\{ h (\zeta (t) )  \}_{0\le t \le 1} $  satisfies the first three conditions and 
is volume preserving. 
Relabel $h(\zeta(t))$ as $h(t)$. To achieve  condition (iv), we make use of diffeomorphisms on $S^n$.
Let $\{\phi_t\}_{0 \le t \le 1}$  be a $1$-parameter family of diffeomorphisms  on $S^{n}$ 
generated by a $t$-dependent, smooth vector field $X_t = X( \cdot, t)$   to be chosen later. 
Define $g(t)=\phi_t^*(h(t))$.  Let  $ \dot g = \frac{d}{dt } g$, then  
\be 
\frac{d}{dt}dV_{g(t)}=\dfrac{1}{2}\tr_g\dot{g}\, dV_{g(t)} 
\ee
and
\be
\dot{g}=\dfrac{d}{dt}\phi_t^* (h(t) ) =\phi_t^*\left(\dfrac{d}{dt}h (t) \right)+\phi_t^*(\L_{X_t}  h(t)). 
\ee
Hence, 
\be \label{vgt}
\begin{split}
\tr_g\dot{g} = & \ \tr_{\phi_t^*(h (t) )}\left( \phi_t^*\left(\dfrac{d}{dt}h\right)+\phi_t^*(\L_{X_t}  h(t))\right) \\
= & \ \phi_t^*\left( \tr_h\dot{h}+\tr_h\L_{X_t}h(t)   \right) \\
= & \ \phi_t^*\left( \tr_h \dot{h}+2\Div_h X_t\right).
\end{split}
\ee
Now 
let $\psi(t,x)$ be  a smooth function on $I \times S^n$  obtained by 
solving  the elliptic equation on $S^n$,
\be\label{elliptic}
\Lap_h \psi(t,\cdot)=-\dfrac{1}{2}\tr_h\dot{h}, 
\ee
 for each $t $. \eqref{elliptic}  is solvable since $\dint_{S^n}\dfrac{1}{2}\tr_h\dot{h}\,dV_{h(t)}=\dfrac{d}{dt}\dint_{S^n}\,dV_{h(t)}=0$. Furthermore, the solution $ \psi (t , \cdot)$ depends smoothly on $t$.
Let $ X_t =\grad^{h(t)}\psi  $, where $ \nabla^{h(t)} $ is the gradient on  $(S^n, h(t))$. 
Clearly, $ \tr_g \dot g = 0 $ by \eqref{vgt} and \eqref{elliptic}. Condition (iv) is thus  satisfied. 
\end{proof}

Next,  given a fixed choice of $\{ h(t) \}$, 
we continue to  denote the path provided  in Lemma \ref{nice-path} by $\{ g(t) \}$.
The following lemma 
deforms the metric $ d t^2 + g(t)$ on $I\times S^{n}$ to a metric of  positive scalar curvature.

\begin{lemma}\label{collar}
There exists 
$A_0>0$ such that for all 
$\veps  \in [0,1]$ and $A\geq A_0$, 
the metric on $[0,1]\times S^{n}$ given by
\be \label{collar-metric}
\gamma_{\veps}=A^2dt^2+(1+\veps t^2)g(t),
\ee
has positive scalar curvature on $I \times S^n$, 
$\{0\}\times S^{n}$ is minimal, and the spheres $\{t\}\times S^{n}$ for $t\in (0,1]$
 are mean convex with respect to the normal direction $\pr_t$.
\end{lemma}
\begin{proof}
Consider a metric of the form
\be
\gamma=A^2dt^2+h(t),
\ee
where 
$h(t)=(1+\veps t^2)g(t)$, $A>0$ and $\veps>0$, to be determined later (here we are abusing notation by using $h(t)$ again).
Direct calculations give
\be
\begin{split}
R(\gamma)&=R(h) +A^{-2}\left[-\tr_h\ddot{h}-\frac{1}{4}(\tr_h \dot{h})^2+\frac{3}{4}|\dot{h}|^2_h\right] \\
&\geq \inf\limits_{t,x}R(h)+A^{ -2}\left[-\dfrac{2n\veps}{(1+\veps t^2)^{ -1}} -\sup\limits_{t,x}|\tr_g\ddot{g}| -\dfrac{n^2 \veps^2 t^2}{(1+\veps t^2)^2} \right]. \end{split}
\ee
Hence, by picking $A_0 \gg 1$ sufficiently large and $A\geq A_0$
, the metric
\be
\gamma_{\veps}=A^2dt^2+(1+\veps t^2)g(t)
\ee
has positive scalar curvature for all $\veps\in[0,1]$. Note that the mean curvature of any slice $ \{ t \} \times S^{n}$, is given by
\be
H_t=\dfrac{n\veps t}{A(1+\veps t^2)} .
\ee 
Therefore, $H= 0 $  when  $t=0$ and $ H > 0 $ when $ t > 0$. 
\end{proof}

\begin{remark}
Since we have the stronger condition $R(g(t))> 0 $, we do not need
to use the first positive eigenfunction of the operator $-\Delta_g + \frac12 R(g)$
as a warping factor in \eqref{collar-metric} as opposed to that being used in
\cite{schoen-mant}.

\end{remark}

\subsection{Bending the Schwarzschild metric}

The $(n+1)$-dimensional spatial Schwarzschild manifold (outside its horizon) is given by 
\bee
( M^{n+1}, g_m ) = \left( [r_0, \infty) \times S^n,  \frac{1}{1 - \frac{2m}{r^{n-1} } } d r^2 + r^2 g_* \right) ,
\eee
where $g_*$ denotes  the standard metric on $S^{n}$ with constant sectional curvature $1$ and $r_0 = \left( 2 m \right)^\frac{1}{n-1} $.
Replacing $ r$ by $s$, which is the distance function to the horizon $ \{ r_0 \} \times S^n $, 
 one can re-write $g_m$ as 
\be
g_m=ds^2+u_m(s)^2g_*,
\ee
defined on $[0,\infty)\times S^{n}$. Here the horizon $\{ r  = r_0 \}$ corresponds to $s=0$. The function $u_m (s) $ satisfies
\begin{enumerate}[(a)]
\item $u_m(0)=(2m)^{\frac{1}{n-1}}$,
\item $u_m'(0)=0$,
\item $u_m'(s)=\left(1-\dfrac{2m}{u_m(s)^{n-1}}\right)^{1/2}$ for $s>0$, and
\item $u_m''(s)=(n-1)\dfrac{m}{u_m^{n}}$ for $s>0$.
\end{enumerate}
In particular, when $n=3$,
\be
u_m(0)=\sqrt{2m}, \ \  u_m'(0)=0, 
\ee
and 
\be
\, u_m'(s)=\left(1-\dfrac{2m}{u_m(s)^2}\right)^{1/2}, \ \  u_m''(s)=\dfrac{2m}{u_m^3(s)},\quad\textmd{ for $s>0$}.
\ee

The next Lemma ``bends"  the metric $g_m$  near the horizon $\{ s = 0 \}$ so that the resulting metric has 
strictly positive scalar curvature near $\{ s = 0 \}$.

\begin{lemma}  \label{lma-bending}
Let $s_0>0$. There exist a small $\delta>0$ and a smooth function $\sigma:[s_0-\delta,\infty)\to (0,\infty)$ satisfying
\begin{enumerate}
\item $\sigma(s)=s$ for all $s\geq s_0$,
\item $\sigma$ is monotonically increasing, and
\item the metric $ds^2+u_m(\sigma(s))^2g_*$ has positive scalar curvature for $s_0-\delta\leq s< s_0$ and vanishing scalar curvature for $s\geq s_0$.
\end{enumerate}
\end{lemma}

\begin{proof}
Recall that for a metric 
$
\tilde{g}=ds^2+f(s)^2g_*,
$
its scalar curvature is given by
\be
\tilde{R}=nf^{-2}\left[ (n-1) - (n-1)\dot{f}^2 -2f\ddot{f} \right].
\ee
Hence for $ds^2+u_m(\sigma(s))^2g_*$, we need to have 
\be
\tilde{R}=nu_m^{-2}\left[ (n-1) - (n-1)\left(\dfrac{d}{ds}u_m\right)^2 -2u_m\dfrac{d^2}{ds^2}u_m \right]>0,
\ee
on $[s_0-\delta,s_0)$. Thus, it  is sufficient  to require 
\be
(n-1) - (n-1)\left(\dfrac{d}{ds}u_m(\sigma)\right)^2 -2u_m\dfrac{d^2}{ds^2}u_m((\sigma))>0,
\ee
on $[s_0-\delta,s_0)$.
By  the fact that Schwarzschild is scalar flat,  this  is equivalent to 
\be\label{scal-flat}
(n-1)-(n-1)(\sigma')^2-2u_m(\sigma)u_m'(\sigma)\sigma''>0,
\ee
on $[s_0-\delta,s_0)$.
Now  define $\theta(s)=1+e^{-\frac{1}{(s-s_0)^2}}$ and $\theta(s_0)=1$. 
For  sufficiently small $\delta$,  let
\be
\sigma(s)=\dint_{s_0-\delta}^{s} \theta(s)\,ds+K_{\delta},
\ee
where $K_{\delta}$ is chosen so that $\sigma(s_0)=s_0$ and thus can be extended to be equal to $s$ for $s \geq s_0$. 
With such a choice of $\sigma(s)$, (\ref{scal-flat}) becomes
\be
\begin{split}
& \ (n-1)-(n-1)[1+2e^{-\frac{1}{(s-s_0)^2}}+e^{-\frac{2}{(s-s_0)^2}}]-4u_m(\sigma)u_m'(\sigma)\dfrac{e^{-\frac{1}{(s-s_0)^2}}}{(s-s_0)^3} \\
= & \ e^{-\frac{1}{(s-s_0)^2}}\left(-2(n-1)-(n-1)e^{-\frac{1}{(s-s_0)^2}}-4u_m(\sigma)u_m'(\sigma)\dfrac{1}{(s-s_0)^3}\right).
\end{split}
\ee
By taking $\delta$ sufficiently small, this last quantity is positive.
\end{proof}

\subsection{Gluing lemma}

The following lemma allows one to glue a  collar extension $(I \times S^n, \gamma_\epsilon)$ from Lemma \ref{collar} to the 
``bending" of the Schwarzschild metric in Lemma \ref{lma-bending}.

\begin{lemma}\label{pasting}
Let $ g_*$ be the standard metric  (of constant sectional curvature $1$) on $S^{n}$.
Let $f_i:[a_i,b_i]\to \R$, $ i =1, 2$,  be  two 
smooth
functions satisfying 
\begin{enumerate}[(I)]
\item $f_i>0$, $f_i'>0$ and $f_i''>0$ on $[a_i,b_i]$,
\item the metric $dt^2+f_i(t)^2g_{*}$ on $[a_i,b_i]\times S^{n}$  has positive scalar curvature,
\item $f_1(b_1)<f_2(a_2)$ and $f_1'(b_1)=f_2'(a_2)$.
\end{enumerate}
Then, after translating the intervals so that $a_2-b_1=(f_2(a_2)-f_1(b_1))/f_1'(b_1)$, there exists a smooth  function $f:[a_1,b_2]\to\R$ satisfying
\begin{enumerate}[(i)]
\item $f>0$ and $f'>0$ on $[a_1,b_2]$,
\item $f=f_1$ on $[a_1,\frac{a_1+b_1}{2}]$,
\item $f=f_2$ on $[\frac{a_2+b_2}{2},b_2]$, and
\item the metric $dt^2+f(t)^2g_*$  on $[a_1,b_2]\times S^{n}$ has positive scalar curvature.
\end{enumerate}
\end{lemma}

\begin{proof}
Define $\tilde f $ on $[a_1, b_2]$ so that $\tilde f $ agrees with $f_1$ and $f_2$ on $[a_1, b_1]$ and $[a_2, b_2]$, respectively, 
and whose graph on $[b_1, a_2]$ is the line segment connecting $(b_1, f_1(b_1))$ and $(a_2, f_2 (a_2) )$.
Clearly,  $\tilde{f}\in C^{1,1}([a_1,b_2])$ and $\tilde{f}$ is smooth  except at $b_1$ and $a_2$. 

Define  $m_i=(a_i+b_i)/2$, $ i =1, 2$.  Let $\delta>0$ be such that $m_1<b_1-\delta$ and $a_2+\delta<m_2$. Let $\eta_{\delta}$ be a smooth cut-off function such that $\eta_{\delta}(t)=1$ on $[b_1-\delta,a_2+\delta]$ and $\eta_{\delta}(t)=0$ on $[a_1,m_1]\cup[m_2,b_2]$. Define the following mollification of $\tilde{f}$:
\be
f_{\nu}(t)=\int_{\R}\tilde{f}(t-\nu\eta_{\delta}(t)s)\phi(s)\,ds.
\ee
This mollification fixes $\tilde{f}$ on $[a_1,m_1]\cup[m_2,b_2]$ and coincides with the standard mollification on an interval properly containing $[b_1,a_2]$; on the remaining part, its value is given by a standard mollification of $f$ with radius $\nu\eta_{\delta}(x)\leq\nu$. It can be easily checked that both $f_{\nu}\to f$ and $f_{\nu}'\to f'$ in $C^0([a_1,b_2])$, as $\nu\to 0$.

A direct calculation shows  that for $f>0$, the metric $\tilde{g}=dt^2+f(t)^2g_*$ has positive scalar curvature if and only if 
\be
f''(t)<\dfrac{(n-1)}{2f(t)}\left( 1-f'(t)^2\right).
\ee
Thus, by assumption $(II)$,
\be
f_i''(t)<\dfrac{(n-1)}{2f_i(t)}\left( 1-f_i'(t)^2\right),\,\textmd{ on $[a_i,b_i]$.}
\ee
The condition $f_i''>0$ on $[a_i,b_i]$ ensures that the graph of
\be
\Omega[\tilde{f}](x)=\dfrac{(n-1)}{2\tilde{f}(t)}\left( 1-\tilde{f}'(t)^2\right)
\ee
lies strictly above the graph of $\tilde{f}''$ (when defined) and the graphs of $f_1''$ and $f_2''$. 
Clearly, $\Omega[f_{\nu}]\to \Omega[\tilde{f}]$ in $C^0([a_1,b_2])$ as $\nu\to 0$. 
Let $3d$ be the smallest vertical distance from the graph of $\Omega[{\tilde{f}}]$ 
to the graphs of $f_1''$ and $f_2''$. 
The uniform convergence imply that we can take a small $\nu$ so that the graph of $\Omega[f_{\nu}]$ lies 
exactly within a distance $d$ from the graph of $\Omega[\tilde{f}]$. 
Since $\Omega[\tilde{f}]$ is uniformly continuous, 
there exists a number $\nu>0$ such that $\Omega[\tilde{f}](s)\leq \Omega[\tilde{f}](t)+d$ 
on $[t-\nu,t+\nu]$. For simplicity, abusing notation, set $\tilde{f}''(b_1)=f''_1(b_1)$ 
and $\tilde{f}''(a_2)=f_2''(a_2)$. Then it follows that for a sufficiently small $\nu$,
\be
f_{\nu}''(t)\leq \sup\limits_{[t-\nu,t+\nu]} \tilde{f}''(s)+d\leq \sup\limits_{[t-\nu,t+\nu]} \Omega[\tilde{f}](s)-3d+d\leq \Omega[\tilde{f}](t)-d,
\ee
and hence $f_{\nu}''(t)< \Omega[f_{\nu}](t)$ on $[a_1,b_2]$. It follows that the metric $dt^2+f_{\nu}(t)^2g_*$ has positive scalar curvature.
\end{proof}

\section{Proofs of Theorem \ref{main} and Theorem \ref{main-2}}  \label{sec-proof}
With the  paths of metrics $\{ h(t) \}_{0\le t \le 1}$  in $ \Scal^+(S^3)$ and  $ \Scal^+_* (S^n)$ given in Sections \ref{ricci-path} and \ref{star-shaped},  respectively,  
one can prove Theorem \ref{main} and \ref{main-2} in the same way that Theorem 2.1 was proved in \cite{schoen-mant}.

\begin{proof}[Proof of Theorem \ref{main}]
Let $g\in \Scal^+(S^3)$ and let $\{ h(t) \}_{0 \le t \le 1} $ be given by Corollary \ref{coro-smooth}.  
Normalize this path so that it is volume preserving by considering 
\be
\tilde{h}(t)=\psi(t)h(t),\,\textmd{with}\,\, 
\psi(t)=\left( \frac{\vol(g)}{\vol(h(t))}  \right)^{\frac{2}{3}}.
\ee
Apply Lemma \ref{nice-path} to $\{ \tilde{h}(t) \}_{0 \le t \le 1}$ to obtain $\{ g(t) \}_{0 \le t \le 1}$. 
Let $m>0$ be a constant such that $\omega_3 (2 m)^{3/2}>\vol(g)$. 

In what follows, we set $ n = 3 $, though  we will keep using the notation $n$ to emphasize that this part of the proof holds in general dimensions.
Consider the family of collar extensions obtained in Lemma \ref{collar}, i.e., the metrics
\be
\gamma_{\veps}=A^2dt^2+(1+\veps t^2)g(t)
\ee
with positive scalar curvature on $[0,1]\times S^n$ for 
 $\veps\in[0,1]$. Let $g_*$ denote the round metric on $S^{n}$. Then $g(t)=\rho^2 g_*$ for some $\rho>0$
on $[1/2,1]$ (recall that $g(1)$ is round and $g(t) = g(1) $ for $t \in  [1/2,1]$). 
Make the change of variables 
$s=At$ on $[1/2,1]$, obtaining
\be
\gamma_{\veps}=ds^2+(1+\veps A^{ -2}s^{2})\rho^2 g_*,
\ee
for $s\in [A/2,A]$.

Define $f_{\veps}(s)=(1+\veps A^{ -2}s^{2})^{1/2}\rho$. Then,
\begin{align}
f_{\veps}'(s)&=\dfrac{\rho \veps s}{A^2(1+\veps A^{ -2}s^2)^{1/2}}>0, \\
f_{\veps}''(s)&=\dfrac{\rho \veps}{A^2(1+\veps A^{ -2}s)^{3/2}}>0.
\end{align}

This function will play the role of $f_1$ in Lemma \ref{pasting}. The role of $f_2$ will be played by the function $u_m(\sigma(s))$ in the Schwarzschild bending $ds^2+u_m(\sigma(s))^2g_*$ from Lemma \ref{lma-bending}. To be able to apply Lemma \ref{pasting} we need 
$f_{\veps}(A)<u_m(\sigma(s_0-\delta))$ and 
$f_{\veps}'(A)=u_m'(\sigma(s_0-\delta))$; to achieve this condition we will choose $\veps$ and $\delta$ accordingly. Consider the curves 
$\Gamma(\veps)=(f_{\veps}(A),f_{\veps}'(A))$ and $\Lap(s)=(u_m(s),u_m'(s))$.

Notice that as $\veps\to 0$ 
\be
\Gamma(\veps)\to\left(\rho,0\right)=\left(\left( \dfrac{\vol(g)}{\omega_{n}} \right)^{\frac{1}{n}},0\right).
\ee
Moreover, the slope 
$f_{\veps}'(A)/f_{\veps}(A)$ is strictly decreasing.

When $s\to 0$, $\Lap(s)\to ((2m)^{1/(n-1)},0)$. Since $m$ is chosen so that 
\be
m > \dfrac{1}{2}\left(\dfrac{\vol(g)}{\omega_{n}}\right)^{(n-1)/n},
\ee
$\Lap(0)$ lies to the right of $\Gamma(0)$. Using continuity, pick $s_0$ so the segment of the curve  $\Lap(s)$ from $0$ to $s_0$ lies strictly to the right and below the curve $\Gamma(\veps)$. Apply Lemma \ref{lma-bending} with $\delta$ sufficiently small so that the curve $u_m(\sigma(s)):[s_0-\delta,s_0]\to \mathbb{R}$ has positive second derivative and the curve $\tilde{\Lap}(s)=(u_m(\sigma(s)),u_m'(\sigma(s)))$ still lies to the right and below $\Gamma(\veps)$. Now pick 
$\veps<1$ so that $\Gamma(\veps)=\tilde{\Lap}(s_0-\delta)$, and apply Lemma \ref{pasting} to construct a positive scalar bridge between the collar extensions and the bending of Schwarzschild. The result follows.
\end{proof}

\begin{proof}[Proof of Theorem \ref{main-2}]
Given $ g \in \Scal^+_*(S^n)$,  let $\{ h(t) \}_{0 \le t \le 1}$ be the path of metrics on $S^n$ constructed in Section \ref{star-shaped}. As above,  
to apply Lemma \ref{nice-path}
we normalize this path so that it is volume preserving by setting
\be
\tilde{h}(t)=\psi(t)h(t),\,\textmd{with}\,\, 
\psi(t)=\left( \frac{\vol(g)}{\vol(h(t))}  \right)^{\frac{2}{n}}.
\ee
Apply Lemma \ref{nice-path} to  this new path $\{ \tilde h (t) \}_{ 0 \le t \le 1 }$  to obtain  $ \{ g(t) \}_{0 \le t \le 1}$.
Let $m>0$ such that $\omega_n(2m)^{n/(n-1)}>\Vol(g)$. 
The  rest of the proof now is the same as 
that of Theorem {\ref{main}} above.
\end{proof}

We finish this paper by pointing out 
a more elementary case in which the conclusion of Theorem \ref{main-2} also holds. 
If  $g$ is a metric of positive scalar curvature on $S^n$ ($n\ge 3)$ 
that is conformal to the standard metric,  say $ g = u^{\frac{4}{n-2}}g_*$ for some  smooth positive function $u$,
then it is straightforward to check that the metric $h(t)=\left[ (1-t)u+t \right]^{\frac{4}{n-2}} g_*$, $t\in[0,1]$, has positive scalar curvature for each $t$.
Hence, by applying the proof of Theorem \ref{main-2} to this path $\{ h(t)\}_{0\le t\le 1}$, one knows that Theorem \ref{main-2} holds for 
such metrics  in the standard  conformal class on $S^n$.

\end{document}